\documentclass[11pt,leqno]{amsart}  
\usepackage[a4paper]{geometry} 
\usepackage[dvipsnames]{xcolor}

 \usepackage{url}
 \usepackage{tikz}
\usepackage{mathrsfs}

\usepackage[all]{xy}
\usepackage{hyperref}

\usepackage{amssymb}

\usepackage{enumerate}
\newtheorem{theorem}{Theorem}

\newtheorem{proposition}{Proposition}[section]
\newtheorem{propdef}[proposition]{Proposition \& Definition}
\newtheorem{lemma}[proposition]{Lemma}
\newtheorem{corollaire}[proposition]{Corollary}

\theoremstyle{definition}
\newtheorem{definition}[proposition]{Definition}
\newtheorem{exemple}[proposition]{Exemple}

\newtheorem{remarque}[proposition]{Remark}

\newcommand{\bl}{\begin{lemma}}
\newcommand{\bp}{\begin{proposition}}
\newcommand{\bt}{\begin{theorem}}
\newcommand{\bc}{\begin{corollaire}}
\newcommand{\be}{\begin{equation}}
\newcommand{\bee}{\begin{equation*}}
\newcommand{\bd}{\begin{definition}}
\newcommand{\bdp}{\begin{definitionproposition}}
\newcommand{\bex}{\begin{exemple}}
\newcommand{\br}{\begin{remarque}}
\newcommand{\bpr}{\begin{proof}}

\newcommand{\el}{\end{lemma}}
\newcommand{\ep}{\end{proposition}}
\newcommand{\et}{\end{theorem}}
\newcommand{\ec}{\end{corollaire}}
\newcommand{\ee}{\end{equation}}
\newcommand{\eee}{\end{equation*}}
\newcommand{\ed}{\end{definition}}
\newcommand{\edp}{\end{definitionproposition}}
\newcommand{\eex}{\end{exemple}}
\newcommand{\er}{\end{remarque}}
\newcommand{\epr}{\end{proof}}

\newcommand{\thmref}[1]{Theorem~\ref{#1}}
\newcommand{\propref}[1]{Proposition~\ref{#1}}
\newcommand{\lemref}[1]{Lemma~\ref{#1}}

\newcommand{\remref}[1]{Remark~\ref{#1}}

\newcommand{\defref}[1]{Definition~\ref{#1}}

\def\R{{\mathbb R}}

\def\ov{\overline}

\newcommand{\cal}[1]{\mathcal #1}

\def\t\tc L{{\widetilde{\mathcal L}}}

\def\1{{\boldsymbol 1}}
\def\gd{{\mathfrak{d}}}
\def\gC{{\mathfrak{C}}}
\def\gH{{\mathfrak{H}}}

\def\tc{{\mathtt c}}

\def\tv{{\mathtt v}}
\def\tw{{\mathtt w}}
\def\R{\mathbb{R}}
\def\Z{\mathbb{Z}}

\def\im{{\rm Im\,}}

\def\id{{\rm id}}

\def\codim{{\rm codim\,}}
\def\pr{{\rm pr}}
\def\tN{{\widetilde{N}}}

\def\rc{{\mathring{\tc}}}

 \newcommand{\menos}{\backslash}

\newcommand{\bi}[2]{{#1}^{^{#2}}}
\newcommand{\dos}[2]{{#1}_{_{#2}}}

\newcommand{\Hiru}[3]{{#1}^{^{#2}}{( #3 )}}
\newcommand{\hiru}[3]{{#1}_{_{#2}}{( #3 )}}
\newcommand{\lau}[4]{{#1}^{^{#2}}_{_{#3}}{\left( #4 \right)}}

\newcommand{\parrn}[1]{\addtocounter{proposition}{1} \medskip {\noindent \bf \theparrafo  \  #1.}}

\newcommand{\IH}{\mathscr H}

\newcommand{\nat}[1]{{#1}^\natural}

\title[Lefschetz duality]{Lefschetz duality for intersection (co)homology}
\date{\today}

\author{Martintxo Saralegi-Aranguren}
\address{Laboratoire de Math{\'e}matiques de Lens\\  
      EA 2462 \\
      Universit\'e d'Artois\\
         SP18, rue Jean Souvraz\\
          62307 Lens Cedex\\
         France}
\email{martin.saraleguiaranguren@univ-artois.fr}

\thanks{The author would like to thank Daniel Tanr for many helpful suggestions and comments in the preparation of this work}



\begin{document}

\begin{abstract} 
We prove the Lefschetz duality for intersection (co)homology in the framework of $\partial$-pesudomanifolds. We work with general perversities and without restriction on the coefficient ring.
\end{abstract}

\maketitle

Given an orientable $n$-dimensional $\partial$-compact pseudomanifold $X$, the following Lefschetz duality
$$
\dos {\cal D} X \colon \lau H * {D\ov p} {X;R} \to \lau H {\ov p} {n-*} {X,\partial X;R}  
$$
has been proved in \cite{MR3046315} (see also \cite{MR3175250,LibroGreg}). In these papers, the coefficient ring is a field and the general perversity $\ov p$ verifies $\ov p \leq \ov t$, where $\ov t$ is the top perversity. The RHS is the intersection homology of $X$. The LHS is the intersection cohomology defined by using $D\ov p$-intersection cochains, where $D \ov p$ is the complementary perversity of $ \ov p$. The operator $\dos {\cal D} X $ is the cap-product by an orientation class of $X$.

The hypothesis $\ov p \leq \ov t$ can be eliminated if we avoid the $\ov p$-allowable simplices contained in the singular part of $X$. This can be done by using the  tame $\ov p$-intersection homology  $\lau {\gH} {\ov p} {*} {X;R}$   (see \cite{CST3,LibroGreg}) for which the Lefschetz duality becomes
\be\label{Lef1}
\dos {\cal D} X \colon \lau H * {D\ov p} {X;R} \to \lau \gH {\ov p} {n-*} {X,\partial X;R}.
\ee
This is proved in \cite{LibroGreg} (see also \cite{MR3046315}), where the hypothesis \emph{$R$ is a field} has been weakened by: \emph{$X$ is a locally $(D\ov{p},R)$-torsion free $\partial$-pseudomanifold}.

 Unfortunately, this result does not extend to a general ring. For example, for the closed cone  $X=\tc \R\mathbb P^3$ we get 
 $
\lau H 2 {\ov p} {X;\Z} = \Z_2\not= 0 =\lau \gH {\ov p} {2} {X,\partial X;\Z} , 
 $
 where $\ov p = D\ov p$ is the perversity taking the value 1 on the apex of $X$.
 In order to eliminate the  torsion condition we use, as it is introduced in \cite{CST77} for the Poincar\'e duality,  the blown-up intersection cohomology $\lau \IH * {\ov p} {X;R}$
 and we get the Lefschetz duality
 \be\label{Lef2}
\dos {\cal D} X \colon \lau \IH * {\ov p} {X;R} \to \lau \gH {\ov p} {n-*} {X,\partial X;R}  ,
\ee
 without restrictions on the coefficient ring $R$ (see \thmref{A}).
Since $\lau \IH  *{\ov p}  {X;R} \cong \lau H * {D\ov p} {X;R}$, in the case of locally $(D\ov{p},R)$-torsion free  $\partial$-pseudomanifolds (see \cite[Theorem F]{CST5}) then we observe that \eqref{Lef2} generalizes the Lefschetz duality  \eqref{Lef1} of \cite{LibroGreg}.

Section 4 is devoted to the proof of this Theorem. 
The method of the proof uses the Poincar\'e duality established in \cite{CST77} for pseudomanifolds
(see also \cite{CST44} by using sheaf theory).
In the same way as a boundary manifold is not a manifold, the space $X$ is not a  pseudomanifold and Poincar\'e duality does not apply directly. Then we refine the stratification of $X$ in order to obtain a pseudomanifold $\nat X$ (see \cite{LibroGreg} for another point of view). This is done in Section 1.
The relationship between tame intersection homology (resp. blown-up intersection cohomology) of $X$ and that of $\nat X$  is done in Section 2 (resp. Section 3).

The transformation of a $\partial$-pseudomanifold $ X$ into a pseudomanifold $\nat X$ by refining the stratification as well as the comparison of  both associated  intersection homologies,  by modifying the perversity on $\nat X$, has been done for the first time on  
\cite{MR3028755} in the category of PL-pseudomanifolds.
This is the key point of the work.

\bigskip

We fix for the sequel a commutative ring $R$ with unity. All  (co)homologies in this work are considered with coefficients in $R$. For a compact topological space $X$, we denote by $\tc X= X \times [0,1]/ X \times \{ 0\}$ the \emph{closed cone} on $X$ and $\rc X = X \times [0,1[/ X \times \{ 0\}$ the \emph{open cone} on $X$. A point of a cone is denoted by $[a,t]$. The apex of these cones is $\tv =[-,0]$.

\tableofcontents

\section{Pseudomanifolds}

We present in this section  the geometrical objets used in this work.

\bd\label{def:espacefiltre} 
A \emph{filtered space} is a  Hausdorff topological space endowed with a filtration by closed sub-spaces 
\bee
\emptyset = X_{-1} \subseteq X_0\subseteq X_1\subseteq\ldots \subseteq X_{n-1} \subsetneq X_n=X.
\eee
The
\emph{formal dimension} of $X$ is $\dim X=n$.

The non-empty  connected components of $X_{i}\backslash X_{i-1}$ are the  \emph{strata} of $X$.
Those of $X_n\backslash X_{n-1}$ are \emph{regular strata}, while the others are  \emph{singular strata.}  
The family of strata of $X$ is denoted by $\cal S_X$ or simply $\cal S$.
The {\em singular set} is $X_{n-1}$, denoted by $\Sigma_X$ or simply $\Sigma$. 
 The {\em formal dimension} of a stratum $S \subset X_i\backslash X_{i-1}$ is $\dim S=i$.
The {\em formal codimension} of $S$ is $\codim S = \dim X -\dim S$.
\ed

\bd \label{indu} A subset $U$ of $X$ can be provided with the \emph{induced filtration,} defined by
$U_i = U \cap X_{i}$,$i\in \{0, \ldots,n\}$. 
If $M$ is a topological manifold, with boundary or not, the \emph{product filtration} defined by 
$\left(M \times X\right) _i = M \times X_{i}$, $i\in \{0, \ldots,n\}$.%

When $X$ is compact, the \emph{cone filtration} of  cones $\rc X$ and $\tc X$ are defined respectively by $(\rc X)_i = \rc X_{i-1}$ and $(\tc X)_i = \tc X_{i-1}$, $i \in \{0, \ldots,n+1\}$. We set $\rc \emptyset =\tc \emptyset = \tv$, the apex of the cones $\rc X$ and  $\tc X$.
In the closed cone $\tc X$  we also consider the following filtration
\bee
 \bigg( \{\tv\} \cup X_0\bigg )
 \subset \cdots \subset  
\bigg(  \tc X_{i-1} \cup   X_{i} \bigg)\subset \cdots \subset 
 \tc X,
\eee
$i \in \{1, \ldots, n\}$, denoted by $\nat{(\tc X)}$. The induced filtrations on $X$ by the filtrations $\tc X$ and 
$\nat{\tc X}$ are the same: the original filtration of $X$.
Notice that we have the equalities:
$
\nat{(\tc X)} \menos X = \rc X =( \tc X )\menos X$, as filtered spaces.
\ed

\smallskip

The more restrictive concept of stratified space provides a better behavior of the intersection (co)homology with regard to  continuous maps.

\bd
A \emph{stratified space} is a filtered space verifying the following frontier condition: 
for any two strata $S, S'\in \cal S_X$ such that
$S\cap \overline{S'}\neq \emptyset$ then $S\subset \overline{S'}$.
\ed

The relation $S \preceq S'$, defined on the set of strata by $S\subset \overline{S'}$, is an order relation (see \cite[Proposition A.22]{CST1}). The notation $S \prec S'$ means $S \preceq S'$ and $S \ne S'$.

\begin{definition}\label{def:applistratifieeforte}
Let $f\colon X \to Y$ be a continuous map between two stratified spaces.
The map $f$ is a  \emph{stratified map,}  if it sends a stratum $S$ of $X$ on a stratum $\bi S f$ of $Y$, $f(S) \subset \bi S f$, verifying $\codim S \geq \codim \bi Sf$.

When $f$ is a homeomorphism we say that $f$ is a \emph{stratified homeomorphism}  if $f,f^{-1}$ are stratified maps.

Notice that any inclusion $\i \colon Y \hookrightarrow X$ is a stratified map if we endow $Y$ with the induced filtration and $\dim Y = \dim X$.  This last condition is equivalent to $Y \not\subset \Sigma_X$. The induced stratification is
$
\cal S_Y = \{ (S \cap Y)_{cc} \ne \emptyset / S \in \cal S _X\},
$
where ${}_{cc}$ denotes \emph{a connected component}.

\end{definition}

Pseudomanifolds, introduced by  Goresky and MacPherson in \cite{MR572580,MR696691}, are filtered spaces having a conical local structure.

\begin{definition}\label{pseudo}
A $n$-\emph{dimensional pseudomanifold} is a filtered space,
$$
\emptyset\subset X_0 \subseteq X_1 \subseteq \cdots \subseteq X_{n-2} \subseteq X_{n-1} \subsetneq X_n =X,
$$
such that, for each $i$, 
$X_i\backslash X_{i-1}$ is a topological manifold of dimension $i$ or the empty set. 
Moreover, for each point $x \in X_i \backslash X_{i-1}$, $i\neq n$, there exist
\begin{enumerate}[(i)]
\item an open neighborhood $V$ of  $x$ in $X$, endowed with the induced filtration,
\item an open neighborhood $U$ of $x$ in $X_i\backslash X_{i-1}$, 
\item an $(n-i-1)$-dimensional compact pseudomanifold $L$ where the open cone, $\rc L$, is provided with the cone filtration,
\item a homeomorphism, $\varphi \colon U \times \rc L\to V$, 
such that
\begin{enumerate}[(a)]
\item $\varphi(u,\tv)=u$, for each $u\in U$, 
\item $\varphi(U\times \rc L_{j})=V\cap X_{i+j+1}$, for each $j\in \{0,\ldots,n-i-1\}$.
\end{enumerate}
\end{enumerate}
The pair  $(V,\varphi)$ is a   \emph{conical chart} of $x$
 and the filtered space $L$ is a \emph{link} of $x$. 

The pseudomanifold $X$ is a \emph{classical pseudomanifold} when $X_{n-1} = X_{n-2}$, that is , when the codimension one strata do not appear.
\end{definition}

In this framework  the notion of boundary appears as follows (see \cite{MR3046315,MR3175250}).

\begin{definition}\label{partial}
A \emph{$n$-dimensional $\partial$-pseudomanifold} is a pair $(X,\partial X)$, where $\partial X \subset X$ is a closed subset,  together
with a filtration on $\emptyset\subseteq X_0 \subseteq \dots \subseteq X_{n-1} \subsetneq X_n =X$ such that

\begin{enumerate}[(a)]
\item $X \menos \partial X$, with the induced filtration $(X \menos \partial X)_i = X_i \menos \partial X$, is an $n$-dimensional
pseudomanifold,

\item the subset $\partial X$, with the induced filtration $(\partial X)_{i-1} = \partial X \cap X_i$, is an $(n -1)$-dimensional pseudomanifold,

\item the subset $\partial X$ has an open collar neighborhood in $X$, that is, a neighborhood $N$ with a
homeomorphism $$\flat \colon N \to\partial X  \times  ]0,1] ,$$ preserving the filtrations of \defref{indu} and sending  
$\partial X$ to $ \partial X \times \{1\} $ by $x \mapsto (x,1)$.
\end{enumerate}

Notice that $X_k \cap \partial X\ne \emptyset$ implies $k>0$.
The subset $\partial X$ is called the {\em boundary } of $X$.
We will often abuse notation by referring to the "$\partial$- pseudomanifold X", leaving
$\partial X$ tacit.
\end{definition}

\parrn{Refinement of $X$} \label{ref}	 
The local structure of a $\partial$-pseudomanifold $X$ can be described as follows.
A point of $X \menos \partial X$ possesses a {\em conical chart} $(V,\varphi)$ in the sense of \defref{pseudo}.
A point $x \in \partial X   \cap (X_k  \menos X_{k-1}) $ possesses a {\em conical chart} $(V,\varphi)$ in the following sense:
\begin{enumerate}[(i)]
\item $V \subset X$ is an open neighborhood of $x$,  endowed with the induced filtration,
\item $\varphi \colon V \to  \R^{n-1} \times ]0,1] $ is a homeomorphism with $\varphi(x) = (0,1)$ (case $k=n$), 
\item $\varphi \colon V \to  \R^{k-1} \times \rc L\times ]0,1] $ is a stratified homeomorphism, where $L$ is  a $(n-k)$-compact  pseudomanifold, and $\varphi(x) = (0,\tv ,1)$ (case $k\in \{1, \ldots n-1 \}$).
\end{enumerate}

When $\partial X \ne \emptyset$ the space $X$ is not a  pseudomanifold (see item (iii)) but it always is a stratified space (this comes directly from  from \cite[Proposition A.22]{CST1}, (a), (b) and (c)).

Next results show how to 
refine the stratification of $X$, by considering on $\partial X$ its natural stratification
(b), in order to obtain on $X$  a pseudomanifold structure, denoted by $\nat X$. We find this construction on \cite{MR3028755} in the category of PL-pseudomanifolds.
This procedure is schematized by the following picture:

\begin{center}
\begin{tikzpicture}

\draw (13,3) node {$ X$} ;
\draw(3,3) node {$\nat X$} ;
\draw(8,3.4) node {$I$} ;
\draw [->]  (6,3)--  (10,3);
\draw(8,0) node {$\partial X$} ;

\draw(13,-1) node {\tiny{$\left\{
\begin{array}{l}
1-dimensional  \ strata: 2\\
2-dimensional \ strata: 3
\end{array}
\right.$}} ;

\draw(3,-1) node {\tiny{$\left\{
\begin{array}{l}
0-dimensional  \ strata: 2\\
1-dimensional  \ strata: 5\\
2-dimensional \ strata: 3
\end{array}
\right.$}} ;

\draw (0,0)-- (.5,0);
\draw (1.5,0)-- (2,0);

\draw (2,0)-- (2.5,0);
\draw (3.5,0)-- (4,0);

\draw (4,0)-- (4.5,0);
\draw (5.5,0)-- (6,0);

\draw (2,0)--  (2,2);
\draw (4,0)--  (4,2);

\draw (10,0)-- (16,0);
\draw (12,0)--  (12,2);
\draw (14,0)--  (14,2);

\draw(4,0) node {$\bullet$} ;
\draw(2,0) node {$\bullet$} ;

\draw(1,1) node {\tiny{$R \menos \partial X$}} ;
\draw(3,1) node {\tiny{$R \menos \partial X$}} ;
\draw(5,1) node {\tiny{$R \menos \partial X$}} ;
\draw(2,2.2) node {\tiny{$S \menos \partial X$}} ;
\draw(4,2.2) node {\tiny{$S \menos \partial X$}} ;

\draw(1,0) node {\tiny{$R \cap  \partial X$}} ;
\draw(3,0) node {\tiny{$R \cap \partial X$}} ;
\draw(5,0) node {\tiny{$R \cap \partial X$}} ;
\draw(2,-.2) node {\tiny{$S \cap \partial X$}} ;
\draw(4,-.2) node {\tiny{$S \cap \partial X$}} ;

\draw(11,1) node {\tiny{$R$}} ;
\draw(13,1) node {\tiny{$R $}} ;
\draw(15,1) node {\tiny{$R $}} ;
\draw(12,2.2) node {\tiny{$S $}} ;
\draw(14,2.2) node {\tiny{$S $}} ;
\end{tikzpicture}
\end{center}

\begin{propdef}\label{YY}
Let $X$ be an $n$-dimensional $\partial$-pseudomanifold.
The filtered espace $\nat X_0 \subset \dots \subset \nat X_n=\nat X$,
defined by
$$
\nat{X}_k =X_k \cup (X_{k+1} \cap \partial X),
$$
is an $n$-dimensional pseudomanifold\footnote{Here, $X_{n+1}=X$.}, called the \emph{refinement} of $X$.
 The identity $I \colon \nat X \to X$ is a stratified map.
\end{propdef}

\begin{proof}
Let us prove that $\nat{X}$ is an $n$-dimensional  pseudomanifold. It is a local question. We proceed by induction on $n$. If $n=0$ then $\nat X=X$ is clearly an $n$-dimensional pseudomanifold.

Since the restriction of both filtrations (that of $X$ and that of $\nat X$) to $\nat X\menos \partial X = X \menos \partial X$ are equal, then it suffices to consider a point $x \in \partial X$ and construct a conical neighborhood in the sense of 1.7.

 Using the open collar neighborhood $N$, we can suppose $X= B\times ]0,1] $, where $B$ is a $(n-1)$- pseudomanifold. The filtration on $X$ is the product filtration
$X_i = B_{i-1} \times  ]0,1] $. 

The point $x$ is of the form $(b,1)$. Now, we can reduce $B$ by considering a conical neighborhood of $b$ in $B$. We distinguish two cases.

\smallskip

(i)  $X= \R^{n-1} \times]0,1] $. Here, the point $x$ is $(0,1)$. The  filtration defining $\nat X$ becomes:
$$\nat X_0 = \ldots = \nat X_{n-2} =\emptyset \subset \nat X_{n-1} =  \R^{n-1}\times 
\{ 1 \}
\subset\nat X_{n} = \R^{n-1} \times  ]0,1] .
$$
Then $\R^{n-1} \times  \tc \{0\} $ is a conical neighborhood of $x$ in $\nat X$.

\smallskip

(ii) $X=\R^k \times \rc L\times  ]0,1] $ where $L$ is a $(n-k-2)$-dimensional compact stratified pseudomanifold, with $k \in \{0, \ldots, n-2\}$. Here, the point $x$ is of the form $(0,\tv,1)$, where $\tv$ is the apex of $\rc L$.
The  refinement $\nat X$ is the product filtration
$
\R^k \times  \left(\rc L \times  ]0,1]\right)^\natural .
$
Without loss of
generality, we can suppose $k=0$, that is, $ X =  \rc L  \times ]0,1] $.

 Since $L$ is a $(n-2)$-dimensional compact  pseudomanifold then the closed cone $ \tc L$ is a $(n-1)$-dimensional $\partial$-pseudomanifold with $\partial \tc L = L$ and, by induction hypothesis, $\nat {(\tc L)}$ is a $(n-1)$-dimensional pseudomanifold.
We end the proof if we  find a  stratified homeomorphism
\be\label{llave}
h \colon \rc \left(\nat{(\tc L) } \right)\to \nat{\left(  \rc L \times ]0,1] \right)}.
\ee
The involved filtrations are
$$
\left( \{\tv \}  \times\{1\} \right) 
 \subset \cdots \subset  
\bigg(\left(   \rc L_{r-2}\times ]0,1] \right) \cup \left( \rc L_{r-1}\times \{ 1 \} \right) \bigg)\subset \cdots \subset 
\left(\rc L \times  ]0,1]  \right), \hbox{ and }
$$
$$
\{ \tw \} \subset \rc \left(\tc L_{r-2}  \cup  L_{r-1} \right) \subset \cdots \subset \rc \left({\tc L } \right),
$$
for $r \in \{ 1, \ldots,n-1\}$, where $\tv $ is the apex of the cones  $\rc L$ and $\tc L$ and $\tw $ is the apex of the cone  $\rc\left(\tc L \right)$. 
This homeomorphism is defined by
\bee
h([[a,x],y]) =
\left\{
\begin{array}{ll}
  ([a, 2xy] ,1 -y)& \hbox{si } x \leq 1/2\\[.3cm]
  ([a,  y],1 -2y(1-x))  & \hbox{si } x \geq 1/2.
\end{array}
\right.
\eee
We verify that previous filtrations are preserved. We proceed in three steps.

\begin{itemize}
\item[+]
The restriction
$$
h \colon \rc  L  \to   \rc L \times \{1\} ,
$$ is
given by
$
h([[a,1],y]) =  ([a,  y],1).
$
Since both induced filtrations are the same: the cone filtration of $\rc L$, then we get that $h$
is a   stratified homeomorphism.

\item[+]
The restriction
$$
h \colon \rc  \{\tv\}   \to  \{\tv \} \times ]0,1]  ,
$$ is
given by
$
h([[a,0],y]) =  ([a,  0],1-y).
$
Both filtrations have just one singular point: $\{\tw\}$ and $ (\tv , 1) $ respectively. Since $h(\tw) = (\tv,1)$ then  $h$
is a   stratified homeomorphism.

\item[+] The restriction
$$
h \colon \rc \left({\tc L }\right)  \menos  \rc\left( \{\tw \} \cup L \right)  \to 
 \ \menos 
 \left(\rc L \times    ]0,1] \right)\left(\{\tv \}  \times ]0,1] \cup \rc L \times  \{1\} \right)
$$
 In fact, this homeomorphism becomes 
the map $h \colon    L  \times  ]0,1[  \times  ]0,1[ \to   L    \times  ]0,1[ \times  ]0,1[  ,$
defined by
$$
h(a,x,y) =
\left\{
\begin{array}{ll}
  (a,1 -y, 2xy )& \hbox{si } x \leq 1/2\\[.3cm]
  (a,1 -2y(1-x),  y)  & \hbox{si } x \geq 1/2.
\end{array}
\right.
$$
 The maps $h$
is clearly a   stratified homeomorphism.
\end{itemize}

\medskip

Let us prove the second part of the Proposition. There are two possibilities for a stratum $T \in \cal S_{\nat X}$:
\begin{itemize}
\item $T = S\menos \partial X$ with $S \in \cal S_X$ and $\codim_{\nat X} T = \codim_X S$, or
\item   $T = (S\cap \partial X)_{cc}$ with $S \in \cal S_X$ and $\codim_{\nat X} T = \codim_X S +1$.
\end{itemize}
In both cases, $I(T) \subset S$ and $\codim_X S \leq \codim_{\nat X} T$. So,  $I$ is a stratified map.
\end{proof}

\begin{remarque}\label{borde}
 The stratification of $\nat X$ is given by
\be\label{Strat}
\cal S_{\nat X} =
\left\{
S \menos \partial X / S \in \cal S_X \right\} 
\cup 
\left\{ (S \cap  \partial X)_{cc}  \ne \emptyset / S \in \cal S_X \right\} 
\ee
 Notice that  $\nat X$ and $X$ are equal as topological spaces. They are different as filtered spaces except if $\partial X = \emptyset$.
Stratifications induced on $\partial X$ from the two filtered spaces $X, \nat X$ are the same. For this reason, we also write $\partial X \subset \nat X$. We have $\nat X \menos \partial X = X \menos \partial X$ as filtered spaces.

\end{remarque}

\section{Tame intersection homology}

The Lefschetz duality establishes an isomorphism between the blown-up intersection cohomology and the  tame intersection homology. In this section, we present the second notion 
(see  \cite{MR2210257,MR2209151,LibroGreg,CST3}). It is an intersection homology  where the intersection chains are not included on the singular part of the pseudomanifold. It uses the strata-depending perversities of MacPherson \cite{RobertSF} instead the original perversities of Goresky and MacPherson \cite{MR572580}. Working with the original perversities we get that the tame intersection homology is the intersection homology of Goresky and MacPherson \cite{MR696691}

%:
%:
%:

\bd\label{def:filteredsimplex}
Let  $X$ be a filtered space.   A
 \emph{regular simplex} is a continuous map, $\sigma\colon\Delta\to X$, where the euclidean simplex $\Delta$ is endowed with a decomposition
$\Delta=\Delta_{0}\ast\Delta_{1}\ast\dots\ast\Delta_{n}$,
called \emph{$\sigma$-decomposition of $\Delta$},
verifying
\begin{itemize}
\item[(a)]  $\sigma^{-1}X_{i} =\Delta_{0}\ast\Delta_{1}\ast\dots\ast\Delta_{i},
$
for each~$i \in \{0, \dots, n\}$. 
\item[(b)] $\Delta_n\ne \emptyset$.
\end{itemize}

 The \emph{perverse degree of } $\sigma$ is  the $(n+1)$-tuple,
$\|\sigma\|=(\|\sigma\|_0,\ldots,\|\sigma\|_n)$,  
where
 $\|\sigma\|_{i}=\dim \sigma^{-1}(X_{n-i})=\dim (\Delta_{0}\ast\cdots\ast\Delta_{n-i})$, 
with the convention $\dim \emptyset=-\infty$.
 \item Given a stratum $S$ of  $X$, the \emph{perverse degree of $\sigma$ along $S$} is defined by \ 
 $$\|\sigma\|_{S}=\left\{
 \begin{array}{cl}
 -\infty,&\text{if } S\cap \im \sigma=\emptyset,\\
 \|\sigma\|_{\codim S}&\text{if } S\cap \im \sigma\ne\emptyset,\\
  \end{array}\right.$$
     
We set $\gd \Delta$ the regular part of the chain $\partial \Delta$. 
We define the cochain $\gd \sigma$ by $\sigma \circ \gd$.
 Notice that $\gd^2=0$.
 We denote by $\hiru \gC * {X;R}$ the chain complex generated by the regular simplices of $X$, endowed with the differential $\gd$.
\ed

\begin{definition}\label{def:perversitegen}
A \emph {perversity on a filtered space} $ X$ is a map
$ \ov {p} \colon \cal S_X \to \Z \cup \{\pm\infty\}$ taking the value ~ 0 on the regular strata.
The  pair $(X, \ov {p}) $ is called a \emph {perverse space.}

The \emph{top perversity} is  the perversity defined by $ \ov {t} (S) = \codim (S) -2 $
  on singular strata. The {\em complementary perversity} of a perversity $\ov p$ is the perversity $D\ov p = \ov t - \ov p$.

Let $f\colon X \to Y$ be a stratified map. The \emph{plull-back} of a perversity $\ov q $ of $Y$ is the perversity $f^*\ov q$ of $X$ defined by $f^*\ov q (S) = \ov q (S^f)$, for each $S \in \cal S_X$. 
\ed

\br\label{percon}
Let $X$ be a stratified space endowed with a perversity $\ov p$.
Consider a subset $Y \subset X$ whose induced filtration has formal dimension equal to  $n$. The pull-back perversity of $\ov p$ relatively to the inclusion $\i \colon Y \hookrightarrow X$ is given by
$
\iota^* \ov p( (S \cap Y)_{cc}) = \ov p (S).
$
For the sake of simplicity we shall write $\iota^* \ov p = \ov p$.

Given a topological manifold $M$, with boundary or not, the canonical projection $\pr \colon  M\times X \to X$ is a stratified map. The pull-back perversity $\pr^*\ov p$ will be also denoted by $\ov p$.

Let us suppose that $X$ is compact. A perversity $\ov p $ on the cone $\rc X$ (resp. $\tc X$) is determinate by a number $\ov p(\tv) \in \Z \cup \{ \pm \infty\}$ and a perversity on $X$, still denoted by $\ov p$. Here $\tv$ is the apex of both cones. These perversities are related by $\ov p(S \times ]0,1[) = \ov p (S)$ (resp. $\ov p(S \times ]0,1]) = \ov p (S)$) for each stratum $S \in \cal S_X$.
\er

\bd
Consider a perverse space  $(X,\ov p)$. A 
simplex
$\sigma\colon\Delta\to X$
 is  \emph{$\ov{p}$-allowable} if
  $$
  \|\sigma\|_{S}\leq \dim \Delta-\codim S+\ov{p}(S),
  $$
   for any stratum $S$. We shall say that $\sigma$ is a \emph{$\ov p$-tame simplex} if $\sigma$ is also a regular simplex.
 A chain $c\in \Hiru \gC * {X;R}$ is said to be
   \emph{$\ov{p}$-tame} if is a linear combination of  $\ov{p}$-tame simplices.
   The chain $c$ is a  \emph{tame $\ov{p}$-intersection chain} if $c$ and $\gd c$ are $\ov{p}$-tame chains.   \ed

We define  $\lau \gC {\ov{p}} * {X;R} $ the complex of $\ov{p}$-tame intersection chains endowed with the differential  $\gd $.
Its homology  $\lau \gH {\ov{p}}{*} {X;R}$  is the \emph{tame $\ov{p}$-intersection homology} of $X$

Main properties of this homology have been developed in \cite{CST3,LibroGreg}.
We have proven in \cite{CST3} that the homology
$\lau \gH {\ov{p}} * {X;R}$
coincides with those of \cite{MR2210257,MR2276609} (see also \cite[Chapter 6]{LibroGreg}).
It is also proved that, in the case where $\ov{p}\leq\ov{t}$, this homology coincides with the original intersection homology
 \cite{MR696691}.

\parrn{Relative tame  intersection  homology}
Let $(X,\ov p)$ be a perverse space. Consider $Y$ be a subset of $X$ endowed with the induced filtration having the same formal dimension. The complex of \emph{relative $\ov{p}$-tame chains} is the quotient
$\lau \gC {\ov p} * {X,Y;R} =  \lau \gC {\ov p} *{X;R}  / \lau \gC  {\ov p}* {Y;R} $. Its homology is the
\emph{relative tame $\ov{p}$-intersection homology}  of the \emph{pervers pair}  $(X,Y,\ov{p})$,
denoted by 
$\lau \gH {\ov{p}} *{X,Y;R}$. We have the long exact sequence:
\begin{equation}\label{equa:suiterelative}
\dots
\to
\lau \gH {\ov p} {k+1} {X,Y;R} 
\to
\lau \gH {\ov p} {k} {Y;R} 
\to
\lau \gH {\ov p} k {X;R} 
\xrightarrow {\pr_*}
\lau \gH {\ov p} k {X,Y;R} 
\to \dots,
\end{equation}
where $\pr \colon \lau \gC {\ov p} * {X;R}  \to \lau \gC {\ov p} * {X,Y;R} $ is the canonical projection

\medskip

\parrn{Perversities relating $X$ and $\nat X$}\label{perve1} One of the two key points for establishing our proof of the Lefschetz duality is the understanding of  the relationship between the tame intersection homology and the blown-up intersection cohomology of $X$ and $\nat X$. This is done by using a particular perversity on $\nat X$.

We consider a $\partial$-pseudomanifold $X$ endowed with a perversity $\ov p$. 
 We fix a partition of the boundary  $\partial X = \partial_1 X\sqcup \partial_2 X$ in two families of connected components.
  Given a perversity $\ov p$ over $X$ we define the perversity $\ov P $ on $\nat X$ in the following way (see \eqref{Strat}):
\bee
\ov P = 
\left\{
\begin{array}{cl}
\infty & \hbox{on the strata included in } \partial_1 X\menos \Sigma_X \\[.2cm]
 -\infty & \hbox{on the strata included in } \partial_2 X\menos \Sigma_X \\[.2cm]
 I^* \ov p & \hbox{on the other strata} 
\end{array}
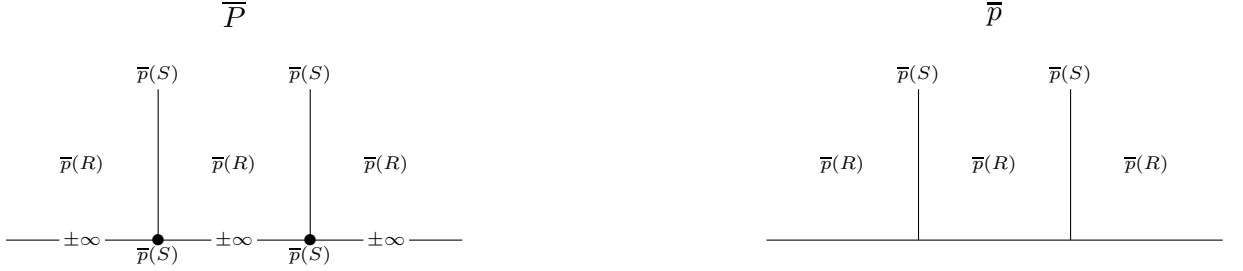
\right.
\eee

\medskip

\begin{center}
\begin{tikzpicture}

\draw (13,3) node {$ \ov p$} ;
\draw(3,3) node {$\ov P$} ;

\draw (0,0)-- (.7,0);
\draw (1.3,0)-- (2,0);

\draw (2,0)-- (2.7,0);
\draw (3.3,0)-- (4,0);

\draw (4,0)-- (4.7,0);
\draw (5.3,0)-- (6,0);

\draw (2,0)--  (2,2);
\draw (4,0)--  (4,2);

\draw (10,0)-- (16,0);
\draw (12,0)--  (12,2);
\draw (14,0)--  (14,2);

\draw(4,0) node {$\bullet$} ;
\draw(2,0) node {$\bullet$} ;

\draw(1,1) node {\tiny{$\ov p(R)$}} ;
\draw(3,1) node {\tiny{$\ov p(R)$}} ;
\draw(5,1) node {\tiny{$\ov p(R)$}} ;
\draw(2,2.2) node {\tiny{$\ov p(S)$}} ;
\draw(4,2.2) node {\tiny{$\ov p(S)$}} ;

\draw(1,0) node {\tiny{$\pm \infty$}} ;
\draw(3,0) node {\tiny{$\pm \infty$}} ;
\draw(5,0) node {\tiny{$\pm \infty$}} ;
\draw(2,-.2) node {\tiny{$\ov p(S)$}} ;
\draw(4,-.2) node {\tiny{$\ov p(S)$}} ;

\draw(11,1) node {\tiny{$\ov p(R)$}} ;
\draw(13,1) node {\tiny{$\ov p(R)$}} ;
\draw(15,1) node {\tiny{$\ov p(R) $}} ;
\draw(12,2.2) node {\tiny{$\ov p(S) $}} ;
\draw(14,2.2) node {\tiny{$\ov p(S) $}} ;
\end{tikzpicture}
\end{center}
 
\medskip

The relationship between the tame intersection homology of $X$ and that of $\nat X$ is given by the following result.
We find this result for the intersection homology on \cite{MR3028755} in the category of PL-pseudomanifolds

\begin{proposition}\label{calculo}
Let $X$ be a $\partial$-pseudomanifold.
We consider $\ov p$ a perversity on $X$.
 Let $\partial X = \partial_1 X\sqcup \partial_2 X$ a partition of the boundary in two families of connected components. 
The composition 
$
\pr \circ I_* \colon \lau \gC {\ov P } * {\nat X;R} \to  \lau \gC {\ov p}*{X,\partial_1 X;R} 
$
is a chain map inducing the isomorphism
$$
\begin{array}{rclcl}
\mathfrak I_1 &\colon &\lau \gH {\ov P } * {\nat X;R} &\to & \lau \gH {\ov p}*{X,\partial_1 X;R} 
\end{array}
$$
\end{proposition}
\begin{proof}
We proceed in two steps.

(I) \emph{The operator $\pr \circ I_*$ is a chain map}.
The operator  $I_*\colon \lau \gC {} * {\nat X;R} \to \lau \gC {} * {X;R}$ is well defined. This comes from \cite[Theorem F]{CST1} and the fact that $I(\nat X\menos \Sigma_{\nat  X}) \subset X \menos \Sigma_X$, since $I$ is  a  stratified map (cf. \propref{YY}). We consider the composition 
$
\pr \circ I_* \colon \hiru \gC * {\nat X;R} \to \hiru \gC * {X,\partial_1 X;R} = \hiru \gC * {X;R}/ \hiru \gC * {\partial_1 X;R}.
$
If we prove 
$$
\sigma \colon \Delta \to \nat X \hbox{ is a } \ov P-\hbox{tame simplex } 
\Rightarrow 
\left\{
\begin{array}{ll}
(a) &I_* (\sigma) \colon \Delta \to X \hbox{ is a  $\ov p$-allowable simplex}.\\[,2cm]
(b) & I_* (\gd \sigma) - \gd I_* (\sigma) \hbox{ is a $\ov p$-tame chain of $\partial_1 X$},
\end{array}
\right.
$$
then we obtain (I).

(a) 
The singular strata of $\nat X$ sent to the regular part of $X$ by $I$ are those of the form:
$(R \cap \partial X)_{cc}$ where $R$ is a regular stratum of $X$.
Outside these strata we have $I^*\ov p = \ov P$. Then, \cite[Proposition 3.6, Remark 3.7]{CST3} says that $I_*(\sigma)$ is a $\ov p$-allowable simplex. We get (a).

\medskip

We prove (b).  Let $\Delta_\sigma =  \Delta_0 * \cdots * \Delta_n = \nabla * \Delta_n$ be the $\sigma$-decomposition of $\Delta$. 
If $|\Delta_n|\geq 1$ or $\nabla = \emptyset$ then we have $\gd \sigma =\partial \sigma$ and (b) comes from $I_* \circ \partial =\partial  \circ I_*$. So, we can suppose $|\Delta_n|=0$ and $\nabla \ne \emptyset$.  Let $\sigma' \colon \nabla \to \nat X$ be  the restriction of $\sigma$. 
From $I_* \circ \partial =\partial  \circ I_*$ we get $I_* (\gd \sigma) - \gd I_* (\sigma)  = (-1)^{|\Delta|}  I_*(\sigma')$. Condition (b) becomes
$$
I_* (\sigma') \hbox{ is included on $\Sigma_X$, or is a $\ov p$-tame simplex of $\partial_1 X$.}$$
We can suppose that the filtered simplex
 $
I_*(\sigma') \colon \nabla\to X$ is regular and  prove that  $
I_*(\sigma') $ is a $\ov p$-allowable simplex of $\partial_1 X$.

 Consider $\{ {\sf S}_0 \prec\cdots \prec {\sf S}_a \}$ the family of strata of $\cal S_{\nat X}$ meeting $\im \sigma$. Since the simplices $\sigma$ and $I_*(\sigma')$ are regular, then ${\sf S}_a = (R \menos \partial X)_{cc}$ and ${\sf S}_{a-1} = (R \cap \partial_1 X) _{cc}$ or $(R \cap \partial_2 X) _{cc}$, for a regular stratum $R \in \cal S_X$.
 
 If ${\sf S}_{a-1} = (R \cap \partial_2 X) _{cc}$ the $\ov P$-allowability of $\sigma$ gives 
$$
0 \leq\dim \nabla = ||\sigma||_{  {\sf S}_{a-1}}  \leq \dim \Delta - \codim_{\nat X} {\sf S}_{a-1}+ \ov P({\sf S}_{a-1} ) 
=
-\infty,
$$
 which is impossible. So, we get
  ${\sf S}_{a-1} =(R \cap \partial_1 X)_{cc}$. 
  The family of strata of 
$\cal S_{\nat X}$ meeting $\im \sigma'$ is 
\be\label{S}
{\sf S}_0 =(S_0 \cap \partial_1 X )_{cc}\prec \cdots  \prec{\sf S}_{a-2} =  (S_{a-2} \cap \partial_1 X )_{cc} \prec {\sf S}_{a-1}=(R \cap \partial_1 X )_{cc},
\ee
where  $\{S_0, \ldots, S_{a-2}\}$ are singular strata of $X$.
This implies that the simplex $I_*( \sigma' )$  lies on $\partial_1 X$.
It remains to prove that $I_*(\sigma')\colon \nabla \to \partial_1 X$ is $\ov p$-allowable. 
The family of strata of $\partial_1 X$ meeting $\im I_*(\sigma')$ is 
$\{
{\sf S}_0 , \ldots, {\sf S}_{a-2}, (R \cap \partial_1 X )_{cc}\}
$
(cf. \eqref{S}). Consider singular stata, that is $j \in \{0,\ldots,a-2\}$, and set $\ell = \codim_{\partial X} {\sf S}_j
= \codim_{\nat X} {\sf S}_j-1$. We have
\begin{eqnarray*}
||I_*(\sigma')||_{  {\sf S}_j} &= &||I_*(\sigma')||_\ell = \dim(\Delta_0 * \cdots * \Delta_{n-\ell})  = ||\sigma||_\ell \\
&\stackrel{\sigma \hbox{ \tiny  is } \ov P-\hbox{\tiny allowable}}{\leq} & \dim \Delta - (\ell+1) + \ov P({\sf S}_j ) 
=
\dim \nabla - \ell -
\ov p({\sf S}_j) \\
 &\stackrel{\tiny Rem. \ 2.3}{=}& \dim \nabla - \ell + \ov p (S_j).
\end{eqnarray*}
This gives (b).

\bigskip

(II) \emph{The operator $\pr \circ I_*$ is a quasi-isomorphism}.
We proceed by induction on $\dim X$.
Notice that the result is clear if $\dim X =0$ or $\partial X = \emptyset$.
Using the open collar neighborhood $N$, the equality $\nat X \menos \partial X = X \menos \partial X$ (see \remref{borde}) and the Mayer-Vietoris sequence  \cite[Proposition 7.10]{CST3}, we can suppose $X= \partial X \times ]0,1] $.
We prove that 
$$\pr \circ I_* \colon
\lau \gH {\ov P } * 
{
\nat{( \partial X \times]0,1] )} ;R 
} \to \lau \gH {\ov p}*{
 \partial X \times   ]0,1] ,\partial_1 X \times \{ 1 \};R}$$ is an isomorphism
 by using \cite[Proposition 2.19]{CST77}. Let us verify the four properties.
Property (i) comes directly from  \cite[Proposition 7.10]{CST3}. Property (ii) is verified since the involved chains have compact support.
Property (iv) is straightforward. Let us see property (iii).

We have two different cases to study (see \ref{ref}.7): $X = \R^{n-1} \times [0,1[$ or $X = \R^i \times \rc L \times ]0,1] $, $L$ 
a compact pseudomanifold.
The perversity $\ov p$ of $X$ is  the  perversity $\ov 0$  in the first case. In the second case, the perversity $\ov p$ comes  from a perversity $\ov p$ defined on $\rc L$. 
The boundary $\partial X =  \R^i \times \rc L \times \{1\}$ is connected. We have two different cases to study: $\partial X = \partial_1 X$ or $\partial X = \partial_2 X$. 

\medskip

Since the $\R^{n-1}, \R^i$-factors commute with the operator $\natural,$ then it suffices  prove that
\begin{itemize}
\item[(j)]
$
 \pr \circ I_* \colon
\lau \gH {\ov P } *
 {\nat{]0,1] }  ;R} 
 \to \lau \gH {\ov 0}*{ ]0,1]  ,\{ 1 \};R}
$
is an  isomorphism where $\ov P (\{1\}) = \infty$.

\smallskip

\item[(jj)]
$
  I_* \colon
\lau \gH {\ov P } *
 {\nat{]0,1] }  ;R} 
 \to \lau \gH {\ov 0}*{ ]0,1] ;R}
$
is an  isomorphism where $\ov P (\{1\}) = -\infty$.

\smallskip
 
\item[(jjj)] $
\pr \circ  I_* \colon
\lau \gH {\ov P } *
 {\nat{(  \rc L \times ]0,1]  )}  ;R} 
 \to \lau \gH {\ov p}*{  \rc L \times ]0,1]  ;\rc L \times \{ 1 \};R}
$
is an  isomorphism for the perversity
$
\ov P = 
\left\{
\begin{array}{cl}
 \infty & \hbox{on } L\menos \Sigma_L \times ]0,1[ \times \{1\} \\[.2cm]
 I^* \ov p & \hbox{on the other strata.} 
 \end{array}  
 \right.
 $
 
  \smallskip

\item[(jjjj)] $
 I_* \colon
\lau \gH {\ov P } *
 {\nat{(  \rc L \times ]0,1]  )}  ;R} 
 \to \lau \gH {\ov p}*{  \rc L \times ]0,1]  ;R}
$
is an  isomorphism for the perversity
$
\ov P = 
\left\{
\begin{array}{cl}
 -\infty & \hbox{on } L\menos \Sigma_L \times ]0,1[ \times \{1\} \\[.2cm]
 I^* \ov p & \hbox{on the other strata.} 
 \end{array}
 \right.
 $

 \end{itemize}

 Let us prove these properties. 
 
  \smallskip
  
  (j) The RHS is $\hiru H * {]0,1], \{ 1 \};R} =0$ (see \cite[Definition 4.8 and Proposition 5.5]{CST3}). On the other hand, $\nat {]0,1]}$ is the open cone $\rc \{1\}$ where the induced perversity is given by $\ov P (\{1\}) = \infty$. Form \cite[Proposition 7.9]{CST3} we get $\lau \gH {\ov P } *
 {\nat{]0,1] }  ;R} =0$.
  
  \smallskip
  
    (jj) The RHS is $\hiru H * {]0,1];R} =\hiru H 0 {]0,1];R}  =R$. On the other hand, since the induced perversity is given by $\ov P (\{1\}) = -\infty$ then, from \cite[Proposition 7.9]{CST3}, we get $\lau \gH {\ov P } *
 {\nat{]0,1] }  ;R} = \hiru H *
{ \{1\}  ;R} =R$.

\smallskip

 (jjj) We know from \eqref{equa:suiterelative} and  \cite[Proposition 7.9, Corollaire 7.8]{CST3} that the RHS is 0. We have that $h_* \colon  \lau  \gH  {h^*\ov P} *{ \rc \left( \nat{(\tc L)}\right) ;R}  \to \lau \gH  {\ov P} * {\nat{(  \rc L \times ]0,1]  )}  ;R}  $
  is an isomorphism since $h$ is a stratified homeomorphism. Since 
  $
  h^* \ov P(\{\tw\}) =  \ov P (\{ \tv\} \times \{1\})  =\ov p (\{ \tv\})
  $
we get 
$
\lau  \gH  {h^*\ov P} *{ \rc \left( \nat{(\tc L)}\right) ;R} 
= 
\lau  \gH  {\ov P} {\leq \ov P (\{ \tv\} \times \{1\})}{ \nat{(\tc L)} ;R} 
\stackrel{\tiny induction}{=}\lau  \gH  {\ov p} {\leq \ov p (\{ \tv\})}{ \tc L,L \times \{1\} ;R} =0
$
 from \eqref{equa:suiterelative} and  \cite[Proposition 7.9, Corollaire 7.8]{CST3}.
 \medskip
 
 (jjjj) 
Let us consider the diagram of filtered spaces:
\be \label{diag}
\xymatrix{
& \nat{(\rc L \times ]0,1]  )}\ar[rr]^I & &
\rc L \times ]0,1] &\\
\rc \left( \nat{(\tc L)}\right) \ar[ru]^h& 
\nat{(\tc L)}  \ar[l]_-{\i} \ar[r]^{I_1}&
\tc L 
 & \rc L \ar[l]_{\iota} \ar[u]_J \ar@/_{10mm}/[ll]_{\iota_1} & L, \ar[l]_{\i_1}
}
\ee
where $I$, $I_1$ are defined in \propref{YY}, $h$ is  defined by \eqref{llave}, $\iota$, $\iota_1$ are  inclusions, $\i (x) = \i_1 (x) = [x,1/2]$ and $J(x) = (x,1/2)$. 
We have the equalities $\iota = I_1 \circ \iota_1$ and $I \circ h \circ \i \circ \iota_1 \circ \i_1 = J  \circ \i_1$.

The notation $\ov p$ designes a perversity on: $\rc L \times ]0,1[$,  $\rc L \times ]0,1]$, $\rc L$, $\tc L$ and $L$ (see \remref{percon}).
There are two refined perversities, denoted by $\ov P$,  on $\nat{(\tc L)}$ and $\nat{(\rc L \times ]0,1]  )}$ respectively. 
We also have the perversity $h^* \ov P$ on $\rc \left( \nat{(\tc L)} \right)$.
They are characterized by:
\be\label{nueve}
\left\{
\begin{array}{ll}
\hbox{On } \nat{(\tc L)} : & 
\left\{
\begin{array}{l}
\ov P = \ov p \hbox { on } \rc L \sqcup \left( \Sigma_L \times \{1\}\right)\\
\ov P = \infty  \hbox { on } \left(  L\menos \Sigma_L \times \{1\}\right)
\end{array}
\right.
 \\[,7cm] 
\hbox{On } \nat{(\rc L \times ]0,1]  )}: & 
\left\{
\begin{array}{l}
\ov P = \ov p \hbox { on } \left( \rc L \times ]0,1[\right) \sqcup \left( \Sigma_{\rc L} \times \{1\}\right)\\
\ov P = \infty  \hbox { on } \left(  \left( \rc L\menos \Sigma_{\rc L}\right) \times \{1\}\right)
\end{array}
\right.
 \\[,7cm] 
\hbox{On } \rc \left( \nat{(\tc L)} \right): & 
\left\{
\begin{array}{l}
h^* \ov P(\{\tw\})= \ov p (\{ \tv\} ) \\
h^*  \ov P = \ov P  \hbox { on } \nat{(\tc L)}
\end{array}
\right.
\end{array}
\right.
\ee
 (see \remref{percon}).
We have
\begin{enumerate}[(A)]

\item $h_* \colon  \lau  \gH  {h^*\ov P} *{ \rc \left( \nat{(\tc L)}\right) ;R}  \to \lau \gH  {\ov P} * {\nat{(  \rc L \times ]0,1]  )}  ;R}  $
  is an isomorphism since $h$ is a stratified homeomorphism.

\item $\i_* \colon \lau  \gH  {\ov P} {\leq \ov p (\{ \tv\})}{ \nat{(\tc L)} ;R} =  \lau  \gH  {h^*\ov P} {\leq \ov p (\{ \tv\})}{ \nat{(\tc L)} ;R}  \to \lau \gH  {h^*\ov P}  *{ \rc \left( \nat{(\tc L)}\right) ;R}  $
is an isomorphism since  \eqref{nueve} and  \cite[Proposition 7.9]{CST3}.

\item $I_{1,*} \colon \lau  \gH  {\ov P} *{ \nat{(\tc L)} ;R} \to \lau  \gH  {\ov p} *{ {\tc L} ;R} $   is an isomorphism by induction hypothesis since we have $\dim \tc L < \dim X$.

\item $\iota_{*} \colon \lau  \gH  {\ov p} *{ {\rc L} ;R} \to \lau  \gH  {\ov p} *{ {\tc L} ;R} $   is an isomorphism as follows.  Consider the open covering $\{ U =  {\tc L} \menos (L \times \{1\}) = \rc L,
 V = \tc L\menos \{\tv \} =
  L \times ]0,1] 
  \}$ of $\tc L$. From    \cite[Corollaire 7.8]{CST3} we know that that the inclusion 
  $
  U \cap V = L \times ]0,1[ \hookrightarrow  L \times ]0,1]=V
   $
   induces an isomorphism in homology. Mayer-Vietoris (see  \cite[Proposition 7.10]{CST3} ) gives the claim.
   
\item $J_{*} \colon \lau  \gH  {\ov p} *{ {\rc L} ;R} \to \lau  \gH  {\ov p} *{ {\rc L} \times ]0,1];R}$   is an isomorphism by
 \cite[Corollaire 7.8]{CST3}.
 
 \item $\i_{1,*} \colon \lau  \gH  {\ov p} {\leq \ov p (\{\tv\})} { { L} ;R} \to \lau  \gH  {\ov p} *{ {\rc L} ;R} $   is an isomorphism by  \cite[Proposition 7.9]{CST3}.

\end{enumerate}
We conclude that 
$I_* \colon  \lau \gH  {\ov P} * {\nat{(  \rc L \times ]0,1]  )}  ;R}  \to \lau \gH  {\ov p} * { \rc L \times ]0,1]  ;R}  $
  is an isomorphism.
\end{proof}

\section{Blown-up cohomology}

The cohomology involved in the Lefschetz Duality is the blown-up cohomology, introduced in \cite{CST5,CST77}. We consider in this section a filtered space $X$.

Let $N_{*}(\Delta)$ 
and  $N^*(\Delta)$ denote the  simplicial chain and cochain complexes
of  an euclidean simplex $\Delta$, with coefficients in $R$. 

 \begin{definition}\label{def:thomwhitney}
The \emph{blown-up  complex of $X$ with coefficients in $R$,} $\Hiru \tN*{X;R}$,
is the cochain complex 
 formed by the elements $ \omega $, associating to any regular simplex,
 $\sigma\colon \Delta_{0}\ast\dots\ast\Delta_{n}\to X$,
an element
 $\omega_{\sigma}\in \Hiru \tN * {\Delta} = \Hiru N * {\tc \Delta_0} \otimes \cdots \otimes \Hiru N * {\tc \Delta_{n-1}} \otimes \Hiru N * {\Delta_n}$,  compatible with regular faces. % 
 The differential $\delta\omega \in\Hiru \tN *{X;R}$ is defined by
 $(\delta\omega)_{\sigma}=\delta(\omega_{\sigma})$.
 \end{definition}

\begin{definition}\label{def:degrepervers}
Let $\Delta_0 * \cdots * \Delta_n$ be a regular simplex.
Consider $ \ell \in \{1, \dots, n \} $. %
Associated to any chain $\xi_1 \otimes \cdots \otimes \xi_n \in
\hiru \tN * {\cal H_\ell} = 
\hiru N * {\tc \Delta_0} \otimes \cdots \otimes \hiru N * {\tc \Delta_{n-\ell -1}}
\otimes   \hiru N * {\Delta_{n-\ell }}\otimes \hiru N * {\tc \Delta_{n-\ell + 1}}  \otimes 
\cdots \otimes \hiru N * {\tc \Delta_{n-1}} \otimes \hiru N * {\Delta_n}$ we define 
$
|\xi|_{>n-\ell} = |\xi_{n-\ell+1}| + \cdots + |\xi_n|$, where $|-|$ means degree.
The \emph{$\ell$-perverse degree} of the cochain $\eta \in \Hiru \tN * {\Delta}$ is
is equal to
$$
||\eta||_\ell = \max
\left\{ |\xi |_{>n-\ell}  \ /   \ \xi \in \hiru \tN * {\cal H_\ell} \hbox{ and } \eta (\xi)\ne 0\right\}.
$$
By convention, we set $ \max \emptyset = - \infty $.
\end{definition}

\begin{definition}\label{def:transversedegreeblowup}
Let $\omega$ a cochain of $\Hiru\tN *{X;R}$.
The \emph{perverse degree of $\omega$ along of a singular stratum,} $S \in \cal S$, is equal to
\begin{equation*}\label{equa:perversstrate}
\|\omega\|_{S}=\sup\left\{\|\omega_{\sigma}\|_{\codim S}\mid 
\sigma\colon\Delta\to X \text{ regular with }\sigma(\Delta)\cap S\neq \emptyset\right\}.
\end{equation*}
By convention, we set $ \sup\emptyset = - \infty $.
We denote by $\|\omega\|$  the map associating to any singular stratum $ S $ of $ X  $ the element $\|\omega\|_{S}$ and 0 to a regular stratum.
\end{definition}

\begin{definition}\label{def:admissible}
 Let $(X,\ov{p})$ be a perverse space. A cochain $\omega\in \Hiru \tN *{X;R}$ is \emph{$\ov{p}$-allowable} if
$
\|\omega\|\leq \ov{p}.
$
A cochain $\omega$ is a \emph{$\ov{p}$-intersection cochain} if $\omega$ and its coboundary, $\delta \omega$,
are $\ov{p}$-allowable. 
We denote by $\lau\tN*{\ov{p}}{X;R}$
the complex of  $\ov{p}$-intersection cochains and
 $\lau \IH {\ov{p}} *{X;R}$  its homology, called
 \emph{blown-up  $\ov p$-intersection cohomology} of $X$ with coefficients in~$R$,
for the perversity $\ov{p}$. % 

Let $\cal U$ be an open covering of $X$. The complex $\lau\tN{*,\cal U}{\ov{p}}{X;R}$ is defined as before excepted that a cochain $\omega \in \lau\tN{*,\cal U}{\ov{p}}{X;R}$ is defined only on the regular simplices $\sigma \colon \Delta \to X$ whose image is contained on an element of $\cal U$. 
\end{definition}

\subsection{Relative blown-up intersection  cohomology}
Let $(X,\ov p)$ be a perverse space. Consider $Y$ be a subset of $X$ endowed with the induced filtration and the induced perversity.
 We suppose $\dim X =\dim Y$.
The complex of \emph{relative $\ov{p}$-intersection cochains} is the direct sum
$\lau \tN * {\ov p}{X,Y;R} =  \lau \tN * {\ov p}{X;R}  \oplus \lau \tN {*-1} {\ov p}{Y;R} $, endowed with the derivative $D(\alpha,\beta) = (d\alpha, \i^*\alpha - d\beta)$. 
Its homology is the 
\emph{relative blown-up intersection cohomology} of the perverse pair  $(X,Y,\ov{p})$,
denoted by 
$\lau\IH{*}{\ov{p}}{X,Y;R}$.
We have the long exact sequence associated to the perverse pair $(X,Y,\ov{p})$:
\begin{equation}\label{equa:suiterelative2}
\ldots\to
\lau \IH{i}{\ov{p}}{X;R}
\xrightarrow {\i^*}
\lau \IH i  {\ov{p}} {Y;R} 
\xrightarrow {j^*}
\lau \IH{i+1}{\ov{p}}{X,Y;R}
\xrightarrow {\pr^*}
\lau \IH{i+1}{\ov{p}}{X;R}\to
\ldots,
\end{equation}
where $\pr \colon \lau \tN {*+1} {\ov p}{X,Y;R}  \to \lau  \tN * {\ov p}{X;R} $ is defined by $\pr(\alpha,\beta) = \alpha$ and 
$j \colon \lau \tN * {\ov p}{Y;R}  \to \lau  \tN * {\ov p}{X,Y;R} $ is defined by $j(\beta) = (0,\beta)$.

\bl\label{cov}
Suppose $X$ is a  $\partial$-pseudomanifold. 
Let $Y$ be a union of connected components of the boundary $\partial X$.We consider the open covering
$\cal U = \{ U =\flat^{-1}(Y \times ]0,1]), V = X\menos \flat^{-1}( Y\times ]1/4,1])\}$ of $X$. 
The relative blown-up $\ov{p}$-intersection cohomology  of $(X,Y)$ can be computed with the complex 
$$\lau {\underline N}{ *,\cal U} {\ov p} {X,Y;R} =\{ \omega \in \lau \tN { *,\cal U} {\ov p}{X;R} / \i^*\omega=0\}.
$$
\el
\bpr
The result is clear when $Y=\emptyset$. Let us suppose $Y \ne \emptyset$.
We have seen in \cite[Corollary 9.8]{CST5} that the restriction 
$
\rho_{\cal U} \colon \lau \tN *{\ov p}{X;R} 
\to \lau \tN { *,\cal U}   {\ov p}{X;R} $ induces an isomorphism in homology. So, the restriction 
$
\rho_{\cal U} \oplus \id  \colon   \lau \tN*{\ov p}{X,Y;R} 
\to
\lau \tN  {*,\cal U}  {\ov p}{X,Y;R} =  \lau \tN  {*,\cal U}  {\ov p}{X;R}  \oplus \lau \tN {*-1} {\ov p}{Y;R}
$
also induces an isomorphism in homology.
So, it suffices to prove that the inclusion
$
J \colon \lau {\underline \tN}{ *,\cal U} {\ov p} {X,Y;R}  \to   \lau \tN  {*,\cal U} {\ov p}{X,Y;R} ,
$
defined by $J(\alpha) =(\alpha,0)$, induces an isomorphism in homology.

Consider $h \colon N_Y=\flat^{-1}(Y \times ]0,1]) \to Y$ the map defined by $h(x) = y$ with $\flat (x) = (y,t)$. Notice that $h=\id$ on $Y$.
Consider $a \colon [0,1] \to \R$ a continuous map verifying $a\equiv 0$ on $[0,1/4]$ and $a \equiv 1$ on $[3/4,1]$. Define the continous function $f \colon X \to \R$ by $f\equiv 0$ on $X\menos \flat^{-1}(Y \times ]1/4,1])$ and $f(\flat^{-1}(x,t)) = a(t)$ on $N_Y$.
We consider the associated cochain $\tilde f \in \lau \tN 0 {\ov 0} {X;R}$ 
verifying $\tilde f \equiv 1$ on $Y$ and $\tilde f\equiv 0$ on $X\menos \flat^{-1}(Y \times ]1/4,1])$ (cf. \cite[Definition 10.2]{CST5}).

If $\beta \in \lau \tN *{\ov p} {Y;R}$ the cochain $\tilde f \smile h^* \beta$ lives in $\lau \tN  {*} {\ov p}{N_Y;R}$
(see \cite[Section 4]{CST5}) for the definition of cup product $\smile$).
It  can be extend it to  $\lau \tN  {*,\cal U} {\ov p}{X;R}$ by 0 since it vanishes on
$\flat^{-1}(Y \times ]0, 1/4]) $. It will be denoted also by $\tilde f \smile h^* \beta \in\lau \tN  {*,\cal U} {\ov p}{X;R}$. Notice that $\i^*(\tilde f \smile h^* \beta ) =\beta$.

Let us prove that $J$ induces an isomorphism in cohomology. Consider a cycle $(\alpha,\beta)\in \lau \tN  {*,\cal U} {\ov p}{X,Y;R}$. It is homologous to $(\alpha,\beta) - D(\tilde f \smile h^* \beta,0) = (\alpha -\delta (\tilde f \smile h^* \beta),0) = J(\alpha -\delta (\tilde f \smile h^* \beta)) $ with $\i^*(\alpha -\delta (\tilde f \smile h^* \beta) )= \i^*\alpha -\delta \beta=0$ and $\delta (\alpha -\delta (\tilde f \smile h^* \beta)) =0$. This gives that $J^*$ is an epimorphism. Consider now a cycle $\alpha \in 
\lau {\underline \tN}{ *,\cal U} {\ov p} {X,Y;R}$ and $(\gamma,\eta) \in \lau \tN { *-1,\cal U} {\ov p} {X,Y;R}$ with $D(\gamma,\eta)=J(\alpha)$. As before, $\gamma - \delta (\tilde f \smile \eta) \in \lau {\underline \tN}{ *,\cal U} {\ov p} {X,Y;R}$ and 
$\delta(\gamma - \delta (\tilde f \smile \eta)) = \alpha$. 
This gives that $J^*$ is an monomorphism.
\epr
\medskip

The relationship between the blown-up intersection homology of $X$ and that of $\nat X$ is given by the following result.

\begin{proposition}\label{calculobis}
Let $X$ be a paracompact and separable $\partial$-pseudomanifold.
We consider $\ov p$ a perversity on $X$.
 Let $\partial X = \partial_1 X\sqcup \partial_2 X$ a partition of the boundary in two families of connected components. 
 Consider $\cal U = \{ \flat^{-1}(\partial_2 X \times ]0,1]), X\menos \flat^{-1}(\partial_2 X\times ]1/4,1])\}$ an open covering  of $X$ and 
 $\cal V $ the induced open covering of $\nat X$ by $I$.
The operator
$
I^*  \colon \lau {\underline{\tN}}{*,\cal U}  {\ov p}{X,\partial_2 X;R} 
 \to \lau \tN {*,\cal V}  {\ov P }  {\nat X;R} $
is a chain map inducing the isomorphism
$$
\begin{array}{rclcl}
\mathfrak I_2  & \colon&  \lau \IH * {\ov p}{X,\partial_2 X;R} &\to& \lau \IH * { \ov P }  {\nat X;R} 
\end{array}
$$
\end{proposition}

\bpr
We proceed in two steps.

(I) \emph{The operator $I^*$ is a chain map.} 
Since $I$ is a stratified map the induced map $I^* \colon \lau \tN * { }  {X;R} \to \lau \tN *{}{\nat X;R} $ is a well defined chain map (cf. \cite[Definition 8.2]{CST5} and so 
$I^*\colon \lau {\underline{\tN}}{*,\cal U}  {\ov p}{X,\partial_2 X;R} 
 \to \lau \tN {*,\cal V}  {}  {\nat X;R} $.
  Let us prove that $I^* \colon \lau {\underline{\tN}}{*,\cal U}  {\ov p}{X,\partial_2 X;R}  \to \lau \tN  {*,\cal V} {\ov P}{\nat X;R} $ is well defined. Recall that $\i \colon \partial_2 X\hookrightarrow X$ denotes the natural inclusion.

  It suffices to consider $\alpha \in \lau \tN {*,\cal U} { \ov p}  {X;R}$, with $\i^*\alpha=0$,  and $ S  \in \cal S_{\nat X}$ a singular stratum and to prove:
$
||I^* \omega||_ S  \leq \ov P ( S ).
$ 
We have
$
||I^*\omega||_ S  \leq ||\omega||_{ S^I}
$
(cf. \cite[Theorem A]{CST5}) and $I^* \ov p ( S )\leq  \ov  P ( S )$ if $S$ is not of the form $(R \cap \partial_2 X)_{cc}$ for a regular stratum $R$ of $X$.
So, it remains to prove
$$
||I^* \omega||_{(R \cap \partial_2 X)_{cc}} =-\infty.
$$
Let $\sigma \colon \Delta \to\nat  X$ be a regular simplex with $\im \sigma \cap (R \cap \partial_2 X)_{cc} \ne \emptyset$
and $\im\sigma \subset I^{-1}(\partial_2 X \times ]0,1])$. We need to prove 
$
||(I^*\omega)_\sigma||_1 = -\infty$ since $\codim_{\nat X} (R \cap \partial_2 X)_{cc}=1$.

The family of strata of $\nat X$ meeting $\im \sigma$ is necessarily of the form
\be\label{descomp}
(S_0 \cap \partial_2 X)_{cc} \prec \cdots \prec (S_{b-1} \cap \partial_2 X)_{cc} \prec
(R \cap \partial_2 X)_{cc} \prec R \menos \partial X,
\ee
where $S_0 \prec \cdots \prec S_{b-1}$ are singular strata of $X$ preceding $R$.
So, if $\Delta_0* \cdots * \Delta_n$ is the $\sigma$-decomposition of $\Delta$ we we get  that the $I_*(\sigma)$-decomposition is the elementary amalgamation $\cal A_\sigma \colon \Delta_0* \cdots * \Delta_n\to \Delta_0 * \cdots * \Delta_{n-2} * (\Delta_{n-1} * \Delta_n)$ (cf. \cite[Section 7]{CST5}). 
Notice that $\im I_*(\sigma )\subset \flat^{-1}(\partial_2 X \times ]0,1])$. So,
$
||I^* \omega||_{(R \cap \partial_2 X)_{cc}}= || \cal A_\sigma^* \omega_{I_*(\sigma)}||_1
$
 (see \cite[Definition 8.2]{CST5}).
Following the definition of $||-||_1$ we need to prove that the restriction of the cochain $\cal A_\sigma^* \omega_{I_*(\sigma)}$ to
$\hiru \tN *{\mathcal H_{n-1}} = \hiru N *{\tc \Delta_0} \otimes \cdots \otimes \hiru N *{\tc \Delta_{n-1}} \otimes \hiru N *{\Delta_{n-1} }\otimes \hiru N *{\Delta_n}$ vanishes (cf. \defref{def:degrepervers}).
Let us consider the commutative diagram:
$$
\xymatrix{
\tc \Delta_0 \times \cdots \times \tc \Delta_{n-2} \times \tc \Delta_{n-1} \times \Delta_n
\ar[r]^-{\cal A_\sigma =\id \times \theta}
&
\tc \Delta_0 \times \cdots  \times \tc \Delta_{n-2} \times (\Delta_{n-1} * \Delta_n)
\\
\tc \Delta_0 \times \cdots \times \tc \Delta_{n-2} \times  \Delta_{n-1} \times \Delta_n
\ar[r]^-{\tiny \pi = projection} \ar@{^(->}[u]
&
\tc \Delta_0 \times \cdots  \times \tc \Delta_{n-2} \times \Delta_{n-1} \ar@{_(->}[u]
}
$$
(cf. \cite[Definition 7.3]{CST5}). So, the restriction of $\cal A_\sigma^* \omega_{I_*(\sigma)}$ to
$\hiru \tN *{\mathcal H_{n-1} }$ is the pull-back $\pi^*$ of the restriction of $\omega_{I_*(\sigma)}$ to $\hiru N *{\tc \Delta_0} \otimes \cdots  \otimes \tc \hiru N *{\Delta_{n-2}} \otimes \hiru N *{\Delta_{n-1}}$. Let us consider the simplex $\sigma' \colon \Delta_0 * \cdots * \Delta_{n-1} \hookrightarrow \Delta \xrightarrow \sigma \nat X \xrightarrow I X$. It is a regular simplex whose associated filtration is exactly $\Delta_0 * \cdots * \Delta_{n-1}$ (cf. \eqref{descomp}).
By compatibility with regular faces, the restriction of $\omega_{I_*(\sigma)}$ to $\hiru\tN *{\tc \Delta_0} \otimes \cdots  \otimes \hiru\tN *{\tc \Delta_{n-2}} \times \hiru\tN *{\Delta_{n-1}}$ is exactly $\omega_{\sigma'}$.
 Finally, since $\sigma'$ is a regular simplex of  $\partial_2 X$ and $\i^*\omega =0$  then $\omega_{\sigma'}=0$.

\bigskip

(II) \emph{The operator $ I^*$ is a quasi-isomorphism}. This ends the proof following \lemref{cov}.
We proceed by induction on the dimension of $X$.
Notice that the result is clear if $\dim X =0$ or $\partial X = \emptyset$.
Using the open collar neighborhood $N$, the equality $\nat X \menos \partial X = X \menos \partial X$ (see \remref{borde}) and the Mayer-Vietoris sequence  \cite[Theorem C]{CST5}, we can suppose $X= \partial X \times ]0,1] $.

Using \cite[Proposition 13.2]{CST5}  we prove that 
$$ I^* \colon
\lau \IH * {\ov p}{
 \partial X \times   ]0,1] ,\partial_2 X;R}
 \to
 \lau \IH *{\ov P }  
{
\nat{( \partial X \times]0,1] )} ;R 
} $$ is an isomorphism. We verify the four properties.
Property (i) comes directly from  \cite[Theorem A]{CST5}. Properties (ii) and (iv) are straightforward. Let us see property (iii).

We have two diffrent cases to study (see \ref{ref}.7): $X = \R^{n-1} \times [0,1[$ or $X = \R^i \times \rc L \times ]0,1] $, $L$ being
a compact pseudomanifold $L$.
The perversity $\ov p$ of $X$ is  the  perversity $\ov 0$  in the first case. In the second case, the perversity $\ov p$ comes in fact from a perversity $\ov p$ defined on $\rc L$. 
The boundary $\partial X =  \R^i \times \rc L \times \{1\}$ is connected. We have two different cases $\partial X = \partial_1 X$ or $\partial X = \partial_2 X$. 

\medskip

Since the $\R^{n-1},\R^i$-factors commute with the operator $\natural$ then it suffices  prove that
\begin{itemize}
\item[(j)]
$
 I_* \colon
 \lau \IH * {\ov p}{ ]0,1]  ,\{ 1 \};R}
 \to
 \lau \IH * {\ov P } 
 {\nat{]0,1] }  ;R} 
$
is an  isomorphism where $\ov P (\{1\}) = -\infty$.

\smallskip

\item[(jj)]
$
 I^* \colon
  \lau \IH * {\ov p}{ ]0,1] ;R}
  \to
  \lau \gH * {\ov P } 
 {\nat{]0,1] }  ;R} 
$
is an  isomorphism where $\ov P (\{1\}) = \infty$.

\smallskip
 
\item[(jjj)] $
 I^* \colon
 \lau \IH * {\ov p} {  \rc L \times ]0,1]  ;\rc L \times \{ 1 \};R}
 \to
 \lau \IH * {\ov P } 
 {\nat{(  \rc L \times ]0,1]  )}  ;R} 
$
is an  isomorphism for the perversity
$
\ov P = 
\left\{
\begin{array}{cl}
 -\infty & \hbox{on } L\menos \Sigma_L \times ]0,1[ \times \{1\} \\[.2cm]
 I^* \ov p & \hbox{on the other strata.} 
 \end{array}  
 \right.
 $
 
  \smallskip

\item[(jjjj)] $
 I_* \colon
 \lau \IH * {\ov p} {  \rc L \times ]0,1]  ;R}
 \to
 \lau \IH* {\ov P } 
 {\nat{(  \rc L \times ]0,1]  )}  ;R} 
$
is an  isomorphism for the perversity
$
\ov P = 
\left\{
\begin{array}{cl}
 \infty & \hbox{on } L\menos \Sigma_L \times ]0,1[ \times \{1\} \\[.2cm]
 I^* \ov p & \hbox{on the other strata.} 
 \end{array}
 \right.
 $

 \end{itemize}
 
 \medskip

 Let us prove these properties. 
 
  \smallskip
  
  (j) The LHS is $\Hiru H * {]0,1], \{ 1 \};R} =0$ (cf. \cite[12.5 and Theorem E]{CST5}). On the other hand, $\nat {]0,1]}$ is the open cone $\rc \{\tv\}$ where the induced perversity is given by $\ov p (\{\tv\}) = -\infty$. From \cite[Theorem E]{CST5} we get $\lau \IH * {\ov P } 
 {\nat{]0,1] }  ;R} =0$.
  
  \smallskip
  
    (jj) The LHS is $\Hiru H * {]0,1];R} =\Hiru H 0 {]0,1];R}  =R$. On the other hand, since the induced perversity is given by $\ov p (\{\tv\}) = \infty$ then, from \cite[Theorem E]{CST5}, we get $\lau \IH * {\ov P } 
 {\nat{]0,1] }  ;R} = \Hiru H *
{ \{\tv\}  ;R} =R$.

\smallskip

 (jjj) We know from \eqref{equa:suiterelative2} and  \cite[Theorems D,E]{CST5} that the LHS is 0. We have that $h^* \colon   \lau \IH * {\ov P}  {\nat{(  \rc L \times ]0,1]  )}  ;R}  \to \lau  \IH *  {h^*\ov P} { \rc \left( \nat{(\tc L)}\right) ;R} $
  is an isomorphism since $h$ is a stratified homeomorphism. 
 Using \eqref{nueve} and \cite[Theorem E]{CST5}
we get 
$
\lau  \IH  * {h^*\ov P}  { \rc \left( \nat{(\tc L)}\right) ;R} 
= 
\lau  \IH {\leq \ov  p \{ \tv\} }  {\ov P} { \nat{(\tc L)} ;R} 
\stackrel{\tiny induction}{=}\lau  \IH  {\leq \ov p (\{ \tv\})}{\ov q} { \tc L,L \times \{1\} ;R} =0
$
 from \cite[Proposition 12.7]{CST5}.
 \medskip
 
 (jjjj) 
Let us consider the commutative diagram of filtered spaces \eqref{diag}.
We have
\begin{enumerate}[(A)]

\item $h^* \colon  \lau \IH * {\ov P}  {\nat{(  \rc L \times ]0,1]  )}  ;R}    \to   \lau  \IH  * {h^*\ov P} { \rc \left( \nat{(\tc L)}\right) ;R}$
  is an isomorphism since $h$ is a stratified homeomorphism.

\item $\i^* \colon   \lau \IH * {h^*\ov P}  { \rc \left( \nat{(\tc L)}\right) ;R}  \to \lau  \IH {\leq \ov p(\tv)}  {h^*\ov P} { \nat{(\tc L)} ;R} =  \lau  \IH {\leq \ov p(\tv)}  {\ov P} { \nat{(\tc L)} ;R}  
 $
is an isomorphism since  \eqref{nueve} and \cite[Theorem E]{CST5}.

\item $I^{*}_1 \colon  \lau  \IH * {\ov p} { {\tc L} ;R} \to \lau  \IH *  {\ov P} { \nat{(\tc L)} ;R} $   is an isomorphism by induction hypothesis since we have $\dim \tc L < \dim X$.

\item $\iota^{*} \colon  \lau  \IH * {\ov p} { {\tc L} ;R} \to  \lau \IH  *{\ov p} { {\rc L} ;R} $  
 is an isomorphism with the same proof of \propref{calculo} (D) using  \cite[Theorems C,D]{CST5}.

\item $J^{*} \colon  \lau  \IH  * {\ov p} { {\rc L} \times ]0,1];R}\to \lau  \IH * {\ov p} { {\rc L} ;R}$   is an isomorphism by
\cite[Theorem D]{CST5}.
 
 \item $\i_1^{*} \colon  \lau  \IH * {\ov p}{ {\rc L} ;R} \to \lau  \IH {\leq \ov p (\tv)} {\ov p}  { { L} ;R} $   is an isomorphism by  \cite[Theorem E]{CST5}.
 
\end{enumerate}
We conclude that 
$I^* \colon   \lau \IH  *{\ov p}  { \rc L \times ]0,1]  ;R} \to \lau \IH  *{\ov P}  {\nat{(  \rc L \times ]0,1]  )}  ;R}  $
  is an isomorphism.
  \end{proof}

\section{Lefschetz Duality}

We prove the main result of this work: the Lefschetz Duality. First, we present the orientation issues, following \cite{LibroGreg}. 

\parrn{Orientation and fundamental class}
An $n$-dimensional $\partial$-pseudomanifold $X$ is  {\em R-orientable} if the  pseudomanifold $X \menos \partial X$ is $R$-orientable. When $X$ is compact, associated to an $R$-orientation,  there exists the {\em  fundamental class}  $\dos \Gamma  X \in  \lau H {\ov 0} n {X, \partial X;R}$ (see \cite[Theorem 8.3.3]{LibroGreg}). We fix a relative cycle $\dos \gamma {X} \in \lau \gC {\ov 0} n {X;R}$ representing this class, that is,  $\gd\dos \gamma {X} \in  \lau \gC {\ov 0} {n-1}{\partial X;R}$ and $\dos \Gamma  X =[\pr(\dos  \gamma  X )]$,
where $\pr \colon \lau  \gC {\ov p} * {X;R} \to \lau  \gC {\ov p} * {X,\partial X;R} $ is the canonical projection.

The operator $\pr \circ I_* \colon 
\lau \gH {\ov 0} * {\nat X;R} \to \lau \gH {\ov 0} * {X,\partial X;R} $
 is a well defined morphism (direct consequence of \propref{calculo}). We can relate the fundamental classes of $\nat X$ and $(X,\partial X)$ as follows. 
 
\begin{proposition}\label{clase}
Let $X$ be an $n$-dimensional  $R$-oriented compact $\partial$-pseudomanifold. Then, the  pseudomanifold $\nat X$ is $R$-orientable. Moreover, if
$\Gamma_{\nat X} \in \lau \gH {\ov 0} n {\nat X;R}$ is the fundamental class  of $\nat X$ then
$\pr \circ I_*(\Gamma_{\nat X}) $ is the fundamental class of $X$.
\end{proposition}
\begin{proof}
The orientability of $\nat X$ comes from the fact that $\nat X\menos \partial  X = X \menos \partial X$, as stratified pseudomanifolds. The orientation sheaf ${\cal O}^{\ov 0}$ over $X\menos \partial X = \nat X\menos \partial X$ is therefore the same. We write $\mathfrak o^{\ov 0 }
$ the associated global section.
For the second part, we consider a point $x \in X \menos \partial X$ and we prove that  the restriction of $\pr \circ I_*(\Gamma_{\nat X})$ to $\lau \gH {\ov 0} * {X,X \menos  \{x\};R} = \lau \gH {\ov 0} * {X\menos \partial X,(X \menos \partial X)\menos  \{x\};R}$ \cite[Corollaire 7.12]{CST3} is $ \mathfrak o^{\ov 0}(x)$  \cite[Theorem 8.3.3]{LibroGreg}. This comes from the fact that the restriction of 
$[\gamma_{\nat X}]$ to  $\lau \gH {\ov 0} * {\nat X,\nat X \menos  \{x\}} =\lau \gH {\ov 0} * {\nat X\menos \partial X ,(\nat X\menos \partial X) \menos  \{x\}}$  \cite[Corollaire 7.12]{CST3}  is $\mathfrak o^{\ov 0}(x)$  (cf. \cite[Theorem 8.1.15]{LibroGreg}). 
\end{proof}

\parrn{Duality operator $\dos {\cal D} X$} We consider $\partial X = \partial_1 X\sqcup \partial_2 X$ a partition of the boundary in two families of connected components. 
Using the subdivision operator and the associated homotopy operator of \cite[Proposition 7.10]{CST3} we can suppose that $\gamma_X \in 
 \lau \gC {\ov 0} {n}{ U;R} +  \lau \gC {\ov 0} {n}{V;R},
$
where the covering $\cal U = \{ U,V\}$ is defined in \lemref{cov}, relatively to $Y =\partial_2 X$.

We fix a perversity $\ov p$ on $X$. 
The {\em cup product with the fundamental class} is the homomorphism
$$
\dos {\cal D} X \colon \lau \IH * {\ov p} {X,\partial_2 X;R} \to \lau \gH {\ov p} {n-*} {X,\partial_1 X;R}  
$$
 defined by 
$$
\dos {\cal D} X([\alpha]) =  [\pr(\alpha \frown \dos \gamma { X} )].
$$
where $\alpha$ is a cycle of $\lau {\underline{\tN}}{*,\cal U}  {\ov p}{X,\partial_2 X;R} $ and  $\pr \colon \lau  \gC {\ov p} * {X;R} \to \lau  \gC {\ov p} * {X,\partial_1 X;R} $ is the canonical projection. This operator is well defined since 
$\gamma_X \in 
 \lau \gC {\ov 0} {n}{ U;R} +  \lau \gC {\ov 0} {n}{V;R},
$
 $\pr (\alpha \frown \eta_1) =0$ and 
$\pr (\alpha \frown \eta_2) =0$ ($\alpha$ vanishes on $\partial_2 X$), 
where $\partial\gamma_X = \eta_1 + \eta_2$ is the canonical decomposition relatively to $\partial X = \partial_1 X \sqcup \partial_2 X$.

\begin{theorem}\label{A}
Let $X$ be an $n$-dimensional oriented compact $\partial$-pseudomanifold. Consider a perversity $\ov p$ on $X$. Let $\partial X = \partial_1 X\cup \partial_2 X$ a partition of the boundary in two families of connected components. The cup product with the fundamental class induces the isomorphism
$$
\dos {\cal D} X \colon \lau \IH * {\ov p} {X,\partial_2 X;R} \to \lau \gH {\ov p} {n-*} {X,\partial_1 X;R}  
$$
\end{theorem}
\bpr
We consider the diagram
$$
\xymatrix{
\lau \IH * {\ov p} {X,\partial_2 X;R}  \ar[d]_{\mathfrak I_2} \ar[r]^-{\dos {\cal D} X }& \lau \gH {\ov p} {n-*} {X,\partial_1 X;R}
 \\
\lau \IH * {\ov P } {\nat X;R} \ar[r]^{\dos {\cal D} {\nat X} }& \lau \gH {\ov P } {n-*} {\nat X;R} \ar[u]_{\mathfrak I_1}
 }
 $$
where the vertical arrows and the bottom arrow are isomorphisms (\propref{calculo}, \propref{calculobis} and \cite[Theorem B]{CST77}). 
Notice that $X$ is paracompact and separable. The second assertion comes from the  fact that $X$ is recovered by a finite number  of conical charts (see 1.7) and that each of them is separable. This is clear for the charts of type (ii). For those of type (iii) it suffice to apply induction on $L$.

We end the proof if we show the diagram commutes.

We consider $\nat \gamma \in \lau \gC {\ov 0} n {\nat X;R}$ a generator of the fundamental class $\nat \Gamma$ of $\nat X$,
see \propref{clase}. This Proposition also gives $\xi \in \lau \gC {\ov 0}{n+1} {X;R}$ and $\xi_k \in \lau \gC {\ov 0} n {\partial_k X;R}$, $k=1,2$, with $I_*\gamma_{\nat X} - \gamma_X = \xi_1+ \xi_2 + \partial \xi$.
Let  $\alpha$ be a cycle of $\lau {\underline{\tN}}{*,\cal U}  {\ov p}{X,\partial_2 X;R} $.  Notice that $\alpha \frown \xi_1 \in 
\lau \gC {\ov p} {n-*} {\partial_1 X;R}$ and that  $\alpha \frown \xi_2=0$ since $\alpha$ vanishes on $\partial_2 X$. Then
\begin{eqnarray*}
  \mathfrak I_1  \dos {\cal D}{ \nat X}  \mathfrak I_2( [\alpha])
=
  \mathfrak I_1  \dos {\cal D}{ \nat X} ( [I^*\alpha])
=
  \mathfrak I_1  ( [I^*\alpha \frown \gamma_{\nat X}])
=
[\pr(I_*(I^*\alpha \frown \gamma_{\nat X}))] 
&
\stackrel{ \hbox{\tiny \cite[Theorem A(3)]{CST5}}}{=}
&
\\[,2cm]
[\pr(\alpha \frown I_*\gamma_{\nat X})] 
=
[\pr(\alpha \frown \gamma_X + \alpha \frown \xi_1 + \alpha \frown \xi_2 +\alpha \frown \partial \xi)] 
&=&\\[,2cm]
[\pr(\alpha \frown \gamma_X  \pm \partial (\alpha \frown \xi))] 
=
\dos {\cal D} X([\alpha]).& \qedhere&
\end{eqnarray*}
\epr

\providecommand{\bysame}{\leavevmode ---\ }
\providecommand{\og}{``}
\providecommand{\fg}{''}
\providecommand{\smfandname}{\&}
\providecommand{\smfedsname}{\'eds.}
\providecommand{\smfedname}{\'ed.}
\providecommand{\smfmastersthesisname}{M\'emoire}
\providecommand{\smfphdthesisname}{Th\`ese}


\begin{thebibliography}{10}

\bibitem{CST1}
{\scshape D.~Chataur, M.~Saralegi-Aranguren {\normalfont \smfandname}
  D.~Tanr{\'e}} -- {\og {Intersection Cohomology. Simplicial blow-up and
  rational homotopy.}\fg}, \emph{ArXiv Mathematics e-prints} (2012), To appear
  in Mem. Amer. Math. Soc.

\bibitem{CST3}
\bysame , {\og {Intersection homology. General perversities and topological
  invariance}\fg}, \emph{ArXiv Mathematics e-prints 1602.03009} (2016).

\bibitem{CST5}
\bysame , {\og Blown-up intersection cohomology\fg}, in \emph{An alpine bouquet
  of algebraic topology}, Contemp. Math., vol. 708, Amer. Math. Soc.,
  Providence, RI, 2018, p.~45--102.

\bibitem{CST44}
\bysame , {\og {Blown-up intersection cochains and Deligne's sheaves}\fg},
  \emph{ArXiv Mathematics e-prints 1801.02992} (2018).

\bibitem{CST77}
\bysame , {\og Poincar\'e duality with cap products in intersection
  homology\fg}, \emph{Adv. Math.} \textbf{326} (2018), p.~314--351.

\bibitem{MR2209151}
{\scshape G.~Friedman} -- {\og Superperverse intersection cohomology:
  stratification (in)dependence\fg}, \emph{Math. Z.} \textbf{252} (2006),
  no.~1, p.~49--70.

\bibitem{MR2276609}
\bysame , {\og Singular chain intersection homology for traditional and
  super-perversities\fg}, \emph{Trans. Amer. Math. Soc.} \textbf{359} (2007),
  no.~5, p.~1977--2019 (electronic).

\bibitem{LibroGreg}
{\scshape G.~Friedman} -- {\og Singular intersection homology\fg}, Available at
  \url{http://faculty.tcu.edu/gfriedman/index.html}.
  
  
\bibitem{MR3028755}
{\scshape G.~Friedman {\normalfont \smfandname} E.~Hunsicker} -- {\og
  Additivity and non-additivity for perverse signatures\fg}, \emph{J. Reine
  Angew. Math.} \textbf{676} (2013), p.~51--95.


\bibitem{MR3046315}
{\scshape G.~Friedman {\normalfont \smfandname} J.~E. McClure} -- {\og Cup and
  cap products in intersection (co)homology\fg}, \emph{Adv. Math.} \textbf{240}
  (2013), p.~383--426.

\bibitem{MR572580}
{\scshape M.~Goresky {\normalfont \smfandname} R.~MacPherson} -- {\og
  Intersection homology theory\fg}, \emph{Topology} \textbf{19} (1980), no.~2,
  p.~135--162.

\bibitem{MR696691}
\bysame , {\og Intersection homology. {II}\fg}, \emph{Invent. Math.}
  \textbf{72} (1983), no.~1, p.~77--129.

\bibitem{RobertSF}
{\scshape R.~MacPherson} -- {\og Intersection homology and perverse
  sheaves\fg}, \emph{Unpublished AMS Colloquium Lectures, San Francisco}
  (1991).

\bibitem{MR2210257}
{\scshape M.~Saralegi-Aranguren} -- {\og de {R}ham intersection cohomology for
  general perversities\fg}, \emph{Illinois J. Math.} \textbf{49} (2005), no.~3,
  p.~737--758 (electronic).

\bibitem{MR3175250}
{\scshape G.~Valette} -- {\og A {L}efschetz duality for intersection
  homology\fg}, \emph{Geom. Dedicata} \textbf{169} (2014), p.~283--299.

\end{thebibliography}
\end{document}